\documentclass[11pt,a4paper]{amsart}
\usepackage{amssymb}
\usepackage[arrow,matrix,curve]{xy}

\title[Categorification of a linear algebra identity]
{Categorification of a linear algebra identity and factorization of
Serre functors}
\author{Sefi Ladkani}
\DeclareMathOperator{\Hom}{Hom}
\DeclareMathOperator{\modf}{mod}
\DeclareMathOperator{\per}{per}
\DeclareMathOperator{\rank}{rank}

\newcommand{\ten}{\otimes}
\newcommand{\tenl}{\stackrel{\mathbf{L}}{\otimes}}

\newcommand{\bC}{\mathbb{C}}
\newcommand{\bZ}{\mathbb{Z}}
\newcommand{\cB}{\mathcal{B}}
\newcommand{\cC}{\mathcal{C}}
\newcommand{\cD}{\mathcal{D}}
\newcommand{\cK}{\mathcal{K}}
\newcommand{\cO}{\mathcal{O}}

\newcommand{\gL}{\Lambda}
\newcommand{\eps}{\varepsilon}
\newcommand{\vphi}{\varphi}

\newcommand{\dA}{\cD^b(A)}
\newcommand{\dgL}{\cD^b(\gL)}
\newcommand{\dQ}{\cD^b(Q)}

\newcommand{\bform}[2]{\left\langle {#1}, {#2} \right\rangle}
\newcommand{\Bform}{\bform{\cdot}{\cdot}}

\theoremstyle{plain}
\newtheorem{theorem}{Theorem}[section]
\newtheorem{lemma}[theorem]{Lemma}
\newtheorem{prop}[theorem]{Proposition}
\newtheorem{cor}[theorem]{Corollary}

\theoremstyle{definition}
\newtheorem{defn}[theorem]{Definition}
\newtheorem{rem}[theorem]{Remark}
\newtheorem{example}[theorem]{Example}

\numberwithin{equation}{section}

\begin{document}

\begin{abstract}
We provide a categorical interpretation of a well-known identity from
linear algebra as an isomorphism of certain functors between
triangulated categories arising from finite dimensional algebras.

As a consequence, we deduce that the Serre functor of a finite
dimensional triangular algebra $A$ has always a lift, up to shift, to a
product of suitably defined reflection functors in the category of
perfect complexes over the trivial extension algebra of $A$.
\end{abstract}

\maketitle

\section{Introduction}
%%%%%%%%%%%%%%%%%%%%%%

The general philosophy behind categorification, as explained for
example in~\cite{BaezDolan98}, is that numbers should be interpreted as
sets, sets as categories, equalities as isomorphisms and so on. When
one considers linear operators, the following suggested interpretation
makes sense, see also~\cite{KMS08} for a similar definition.

Given the data of a free $\bZ$-module $V$ of finite rank and linear
maps $f_1, f_2, \dots, f_n, g : V \to V$ satisfying $g = f_1 \cdot f_2
\cdot \ldots \cdot f_n$, a (weak) \emph{categorification} of this data
consists of an abelian or triangulated category $\cB$ whose
Grothendieck group $K_0(\cB)$ is isomorphic to $V$, together with exact
functors $F_i : \cB \to \cB$ and $G : \cB \to \cB$, such that:
\begin{itemize}
\item
$F_1, F_2, \dots, F_n, G$ induce linear maps on $K_0(\cB)$ which, under
the isomorphism with $V$, coincide with $f_1,f_2,\dots,f_n,g$;
\item
There is an isomorphism of functors between $G$ and the composition
$F_1 \cdot F_2 \cdot \ldots \cdot F_n$.
\end{itemize}
When $V$ carries additional structure, such as a bilinear form, it is
preferable that this structure lifts to $\cB$ as well.

\subsection{A linear algebra identity}
%%%%%%%%%%%%%%%%%%%%%%%%%%%%%%%%%%%%%%
The following well-known statement concerns products of reflection-like
matrices defined by a square matrix.

\begin{prop} \label{p:linalg}
Let $B$ be any square $n \times n$ matrix over a commutative ring. Then
\begin{equation} \label{e:prod}
-B_{+}^{-1} B_{-}^T = r^B_1 \cdot r^B_2 \cdot \ldots \cdot r^B_n ,
\end{equation}
where the matrices $B_{+}$ and $B_{-}$ are the upper and lower triangular parts
of $B$, defined by
\begin{align*}
(B_{+})_{ij} =
\begin{cases}
B_{ij} & \text{if $i < j$,} \\
1      & \text{if $i = j$,} \\
0      & \text{otherwise,} \\
\end{cases}
& &
(B_{-})_{ij} =
\begin{cases}
B_{ji}     & \text{if $i < j$,} \\
B_{ii} - 1 & \text{if $i = j$,} \\
0          & \text{otherwise,}
\end{cases}
\end{align*}
(so that $B = B_{+} + B_{-}^T$), and for each $1 \leq i \leq n$, the
square matrix $r^B_i$ is obtained from the identity matrix by
subtracting the $i$-th row of $B$, that is,
\begin{align}
\label{e:rBi}
\left(r^B_i \right)_{st} = \delta_{st} - \delta_{si} B_{it},
&& 1 \leq s,t \leq n .
\end{align}
\end{prop}

This statement originally appeared as an exercise in the book of
Bourbaki~\cite[Ch.~5, \S6, no.~3]{Bourbaki02}, following an argument
presented in Coxeter's paper~\cite{Coxeter51}. Various specific cases
have since then appeared in the literature, including
A'Campo~\cite{ACampo76} in the bipartite case and
Howlett~\cite{Howlett82} in the symmetric case. The general form is
stated and proved in an article by Coleman~\cite{Coleman89}, and an
alternative proof can be found in~\cite{Ladkani08}.

As important special case is when $B = C + C^T$ is the symmetrization
of an upper triangular square matrix $C$ with ones on its main
diagonal. In this case the matrices $r^B_i$ are reflections, and the
proposition implies that
\begin{equation} \label{e:coxprod}
-C^{-1} C^T = r^B_1 \cdot r^B_2 \cdot \ldots \cdot r^B_n .
\end{equation}

This equality provides us with two points of view on the so-called
\emph{Coxeter transformation}. First, as known in Lie theory, it is the
product of the simple reflections, as given by the right hand side
of~\eqref{e:coxprod}. Second, as follows from the left hand side, it
can also be described as the automorphism $\Phi$ satisfying
\[
\bform{x}{y}_C = - \bform{y}{\Phi x}_C
\]
where $\Bform_C$ is the bilinear form defined by the matrix $C$ and $x,
y$ are any two vectors, as known in the representation theory of
algebras, see~\cite{Lenzing99}.

\subsection{Categorical interpretation}
%%%%%%%%%%%%%%%%%%%%%%%%%%%%%%%%%%%%%%%

Our categorical interpretation of equations~\eqref{e:prod}
and~\eqref{e:coxprod} is achieved by using functors on triangulated
categories arising from finite dimensional algebras. In order to state
our result in precise terms, we need to recall a few notions from the
representation theory of finite dimensional algebras.

For a finite dimensional algebra $A$ over a field $k$, denote by $\dA$
the bounded derived category of finite dimensional right $A$-modules,
and by $\per A$ its full triangulated subcategory consisting of all
complexes quasi-isomorphic to perfect complexes, that is, bounded
complexes whose terms are finitely generated projective $A$-modules.

The Grothendieck group $K_0(\per A)$ is free abelian of finite rank,
with a basis consisting of the classes of the indecomposable projective
$A$-modules. It is equipped with a bilinear form induced by the Euler
form
\begin{align*}
\bform{X}{Y}_{A} = \sum_{r \in \bZ} (-1)^r \dim_k \Hom_{\dA}(X, Y[r]) &
& X, Y \in \per A.
\end{align*}

The algebra $A$ is called \emph{triangular} if there exist primitive
orthogonal idempotents $e_1, \dots, e_n$ of $A$ such that $e_i A e_j =
0$ for any $j < i$ and $e_i A e_i \simeq k$ for $1 \leq i \leq n$. The
modules $P_i = e_i A$ then form a complete collection of indecomposable
projectives. Taking their classes as a basis for $K_0(\per A)$, it will be
convenient for us to order them $[P_n], \dots, [P_1]$ and to define the
\emph{Cartan matrix} $C_A$ as the matrix of $\Bform_A$ with respect to
that basis, namely
\begin{align*}
(C_A)_{ij} &= \bform{P_{n+1-i}}{P_{n+1-j}}_A = \dim_k \Hom_A
(P_{n+1-i}, P_{n+1-j}) \\
&= \dim_k e_{n+1-j}Ae_{n+1-i} ,
\end{align*}
so that $C_A$ is upper triangular with ones on its main diagonal.

Similarly, for a (finite dimensional) $A$-$A$-bimodule $M$ we can define
a matrix $C_M$ by
\[
(C_M)_{ij} = \dim_k e_{n+1-j} M e_{n+1-i} ,
\]
and call $M$ \emph{triangular} if $C_M$ is upper triangular, or
equivalently, $e_i M e_j = 0$ for any $j < i$. We have $C_M^T =
C_{DM}$, where $DM$ is the \emph{dual} of $M$, defined as $DM =
\Hom_k(M,k)$.

The \emph{trivial extension} $\gL = A \ltimes DM$ is the $k$-algebra
which has $A \oplus DM$ as its underlying vector space, with the
multiplication defined by $(a,\mu) (a',\mu') = (aa', a \mu' + \mu a')$.
Its indecomposable projectives are in bijective correspondence with
those of $A$, and its Cartan matrix is given by $C_{\gL} = C_A + C_M^T$.
Thus, when $A$ and $M$ are triangular,
$(C_{\gL})_+ = C_A$ and $(C_{\gL})_{-} = C_M$.

\begin{theorem} \label{t:AM}
Let $A$ be a finite dimensional triangular algebra over a field and
let $_AM_A$ be a triangular $A$-$A$-bimodule. Set $\gL = A \ltimes DM$
to be the trivial extension of $A$ with the dual of $M$.

Then there exist, for $1 \leq i \leq n=\rank K_0(\per \gL)$,
triangulated functors $R_i : \dgL \to \dgL$ which restrict to $R_i :
\per \gL \to \per \gL$, such that:
\begin{enumerate}
\renewcommand{\theenumi}{\alph{enumi}}
\item
Each functor $R_i$ induces a linear map on $K_0(\per \gL)$ whose
matrix with respect to the basis of indecomposable projective
$\gL$-modules is $r_{n+1-i}^{C_{\gL}}$, cf.~\eqref{e:rBi}, where
$C_{\gL}$ is the Cartan matrix of $\gL$;

\item
The diagrams of triangulated functors
\begin{equation} \label{e:funcprod}
\xymatrix{
{\per \gL} \ar[d]_{- \tenl_{\gL} A_A}
\ar[r]^{R_1} &
{\per \gL} \ar[r]^{R_2} &
{\dots} \ar[r]^{R_n} &
{\per \gL} \ar[d]^{- \tenl_{\gL} A_A} \\
\dA \ar[rrr]^{- \tenl_A DM_A[1]} & & &
\dA
}
\end{equation}
and
\[
\xymatrix{
\dA \ar[d]_{- \ten_A A_{\gL}} \ar[rrr]^{- \tenl_A DM_A[1]} & & &
\dA \ar[d]^{- \ten_A A_{\gL}} \\
{\dgL} \ar[r]^{R_1} &
{\dgL} \ar[r]^{R_2} &
{\dots} \ar[r]^{R_n} &
{\dgL}
}
\]
commute up to a natural isomorphism of functors.
\end{enumerate}
\end{theorem}

The vertical arrows of~\eqref{e:funcprod} induce an isomorphism
$K_0(\per \gL) \to K_0(A)$ sending projectives to projectives. Thus, by
considering the diagram~\eqref{e:funcprod} at the level of the
Grothendieck groups, we get the following commutative diagram
\[
\xymatrix{
K_0(\per \gL) \ar[d]_{I_n} \ar[r]^{r^{C_{\gL}}_n} &
K_0(\per \gL) \ar[r]^(0.6){r^{C_{\gL}}_{n-1}} &
{\dots} \ar[r]^(0.4){r^{C_{\gL}}_1} &
K_0(\per \gL) \ar[d]^{I_n} \\
K_0(A) \ar[rrr]^{-(C_{\gL})_{+}^{-1} (C_{\gL})_{-}^T} & & & K_0(A)
}
\]
(where $I_n$ is the $n \times n$ identity matrix), which explains why
the theorem can be seen as a categorical interpretation
of~\eqref{e:prod} for $B = C_{\gL}$.

To complement this result, we note that any integral square matrix $B$
with non-negative entries and positive entries on its main diagonal can
be realized as a Cartan matrix of a suitable $\gL$ as in the theorem.
More precisely, there exist a finite dimensional triangular algebra $A$
and a triangular bimodule $M$ over $A$ with $B_+ = C_A$ and $B_- =
C_M$, see Section~\ref{ssec:realize} for the details.

\subsection{Application to Serre functors}
%%%%%%%%%%%%%%%%%%%%%%%%%%%%%%%%%%%%%%%%%%

A triangular finite dimensional algebra $A$ is of finite global
dimension, hence its bounded derived category $\dA$ admits a Serre
functor $\nu_A$ in the sense of Bondal and
Kapranov~\cite{BondalKapranov89}. By a result of
Happel~\cite{Happel88}, it is given by the left derived functor of the
Nakayama functor, $\nu_A = - \tenl_A DA$.
Thus, by taking in Theorem~\ref{t:AM} the bimodule $M$ to be $A$, we deduce
the following result on the Serre functor on $\dA$.

\begin{cor} \label{c:Serre}
Let $A$ be a finite dimensional triangular algebra over a field and
let $T(A) = A \ltimes DA$ be its trivial extension algebra.
Then there exist, for $1 \leq i \leq n=\rank K_0(A)$, triangulated
autoequivalences $R_i$ on $\cD^b(T(A))$ which restrict to
autoequivalences on $\per T(A)$, such that:
\begin{enumerate}
\renewcommand{\theenumi}{\alph{enumi}}
\item
Each autoequivalence $R_i$ induces a linear map on $K_0(\per T(A))$,
whose matrix with respect to the basis of indecomposable projective
$T(A)$-modules is given by the reflection $r_{n+1-i}^B$,
where $B$ is the symmetrization of the Cartan matrix of $A$;

\item
The diagrams of triangulated functors
\begin{align}
\label{e:Serreprod}
\xymatrix{
{\per T(A)} \ar[d]_{- \tenl_{T(A)} A_A}
\ar[r]^{R_1} &
{\per T(A)} \ar[r]^(0.6){R_2} &
{\dots} \ar[r]^(0.4){R_n} &
{\per T(A)} \ar[d]^{- \tenl_{T(A)} A_A} \\
\dA \ar[rrr]^{\nu_A[1]} & & &
\dA
}
\\
\intertext{and}
\notag
\xymatrix{
\dA \ar[d]_{- \ten_A A_{T(A)}} \ar[rrr]^{\nu_A[1]} & & &
\dA \ar[d]^{- \ten_A A_{T(A)}} \\
{\cD^b(T(A))} \ar[r]^{R_1} &
{\cD^b(T(A))} \ar[r]^(0.6){R_2} &
{\dots} \ar[r]^(0.4){R_n} &
{\cD^b(T(A))}
}
\end{align}
commute up to a natural isomorphism of functors.
\end{enumerate}
\end{cor}

Thus, one can lift (a shift of) the Serre functor on $\dA$ to a product
of the ``reflections'' $R_i$ in $\per T(A)$. As before, the
diagram~\eqref{e:Serreprod} can be regarded as a categorical
interpretation of equation~\eqref{e:coxprod} for $C=C_A$, the Cartan
matrix of $A$. This can be done for any upper triangular integral
matrix $C$ with non-negative entries and $1$ on its main diagonal,
see Section~\ref{ssec:realize}.

\subsection{On the proof}
%%%%%%%%%%%%%%%%%%%%%%%%%

Section~\ref{sec:proof} is devoted to the proof of the theorem and its
corollaries. A key ingredient in the proof is the proper definition and
analysis of the functors $R_i$. They are defined, for each $1 \leq i
\leq n$, as the left derived functors of tensoring with a two-term complex
of bimodules,
\[
R_i^{\gL} =
- \tenl_{\gL} \left( \gL e_i \ten_k e_i \gL \xrightarrow{m} \gL \right)
\]
where $m$ denotes the multiplication map and $e_1,\dots,e_n$ are the
primitive orthogonal idempotents.

The functors $R_i^{\gL}$ have already been considered in the works of
Rouquier-Zimmermann~\cite{RouquierZimmermann03} on braid group actions
on derived categories of Brauer tree algebras without exceptional
vertex, and by Hoshino and Kato~\cite{HoshinoKato02} in relation with
constructions of two-sided tilting complexes for self-injective
algebras. When the algebra $\gL$ is symmetric and $\dim e_i \gL e_i =
2$, the functor $R_i^{\gL}$ can be viewed as a twist functor in the
sense of Seidel and Thomas~\cite{SeidelThomas01} with respect to the
$0$-spherical object $e_i \gL$. Our result shows the importance of the
functors $R_i^{\gL}$ for a wider class of algebras $\gL$, which are not
necessarily restricted to be self-injective or symmetric.

In the course of the proof we establish the special case of
Theorem~\ref{t:AM} where the bimodule $M$ is zero, namely that for any
finite-dimensional triangular algebra $A$, the composition $R_n^A \cdot
\ldots \cdot R_2^A \cdot R_1^A$ is isomorphic to zero on $\dA$, see
Section~\ref{ssec:triangular}.

Plugging in this statement the definition of $R_i^A$, we obtain a (finite)
projective resolution of the triangular algebra $A$ as a bimodule over
itself. A similar construction, with relation to Hochschild cohomology
computations, was presented by Cibils in~\cite{Cibils89}.

\subsection{Previous work}
%%%%%%%%%%%%%%%%%%%%%%%%%%
Another categorical interpretation of~\eqref{e:coxprod}, in the realm
of representation theory of quivers, is given by a result of
Gabriel~\cite{Gabriel80}, correcting previous paper by Brenner and
Butler~\cite{BrennerButler76}.

For a quiver $Q$ without oriented cycles, one can consider two exact
autoequivalences on the bounded derived category of its path algebra.
The first is the Auslander-Reiten translation, corresponding to the
left hand side of~\eqref{e:coxprod}, and the second is the so-called
Coxeter functor, which was defined by Bernstein, Gelfand and
Ponomarev~\cite{BGP73} as a product of their reflection functors,
corresponding to the right hand side of~\eqref{e:coxprod}.

In~\cite{Gabriel80}, it is shown that for any quiver whose underlying
graph is a tree, or more generally, does not contain a cycle of odd
length, the Auslander-Reiten translation is isomorphic to the Coxeter
functor, thus interpreting the equality in~\eqref{e:coxprod} as an
isomorphism of functors.

In Section~\ref{sec:quivers} we explain this result in more detail and
compare it with our approach.

Another approach to the factorization of Serre functors for certain
finite dimensional algebras, including ones arising from category
$\cO$ associated to semi-simple complex Lie algebras, is presented
by Mazorchuk and Stroppel in~\cite{MazorchukStroppel08}.

\subsection*{Acknowledgements}
%%%%%%%%%%%%%%%%%%%%%%%%%%%%%%

The results in this paper were first presented at the workshop on Spectral
methods in representation theory of algebras and applications to the study of
rings of singularities that was held at Banff in September 2008. I would like
to thank J.\ A.\ de la Pe\~{n}a, C.~Ringel and C.~Stroppel for their
helpful comments and suggestions.

\section{Proof of the theorem}
%%%%%%%%%%%%%%%%%%%%%%%%%%%%%%
\label{sec:proof}

\subsection{The building blocks -- the functors $R_i^{\gL}$}
%%%%%%%%%%%%%%%%%%%%%%%%%%%%%%%%%%%%%%%%%%%%%%%%%%%%%%%%%%%%
Let $\gL$ be a basic finite dimensional algebra over a field $k$ and
let $P_1, \dots, P_n$ be a complete collection of the non-isomorphic
indecomposable projectives in $\modf \gL$, the category of finite
dimensional right $\gL$-modules. Let $e_1, \dots, e_n$ be primitive
orthogonal idempotents in $\gL$ such that $P_i = e_i \gL$ for $1 \leq i
\leq n$.

Fix $1 \leq i \leq n$ and consider the following complex of
$\gL$-$\gL$-bimodules
\[
C_i = \gL e_i \otimes_k e_i \gL \xrightarrow{m} \gL ,
\]
where $\gL$ is in degree 0 and $m$ is the multiplication map. Taking
the tensor product $- \ten_{\gL} C_i$ yields an endofunctor on the
category $\cC^b(\gL)$ of bounded complexes of finite dimensional right
$\gL$-modules, which induces an endofunctor on its homotopy category
$\cK^b(\gL)$.

Since its terms are projective as left $\gL$-modules, the complex
$C_i$ defines a triangulated functor
\[
- \tenl_{\gL} C_i = - \ten_{\gL} C_i : \dgL \to \dgL .
\]
on the derived category $\dgL$ of $\modf \gL$.
Moreover, as the terms are also projective as right $\gL$-modules, this functor
restricts to a functor
\[
- \tenl_{\gL} C_i = - \ten_{\gL} C_i : \per \gL \to \per \gL
\]
on the triangulated subcategory $\per \gL$ of complexes quasi-isomorphic
to perfect ones (that is, bounded complexes of finitely generated
projectives).

In the sequel, when no confusion arises, we shall denote all the above
functors by $R^{\gL}_i$. These functors were considered by Rouquier and
Zimmermann~\cite{RouquierZimmermann03} in relation with braid group
actions on the derived categories of Brauer tree algebras with no
exceptional vertex, and by Hoshino and Kato~\cite{HoshinoKato02} in
relation with constructions of two-sided tilting complexes for
self-injective algebras.

\begin{lemma} \label{l:RiX}
Let $X \in \modf \gL$. Then
\[
R_i^{\gL}(X) = \Hom_{\gL}(P_i, X) \otimes_k P_i \xrightarrow{ev} X
\]
where $ev$ is the evaluation map $ev : \alpha \otimes y \mapsto
\alpha(y)$.
\end{lemma}
\begin{proof}
Clearly, $X \ten_{\gL} \gL e_i \simeq X e_i \simeq \Hom_{\gL}(e_i \gL,
X)$.
\end{proof}

The Grothendieck group $K_0(\per \gL)$ is a free abelian group on the
generators $[P_1],\dots,[P_n]$ equipped with a bilinear form induced
by the Euler form
\begin{align*}
\bform{X}{Y}_{\gL} = \sum_{r \in \bZ} (-1)^r \dim_k \Hom_{\dgL}(X,
Y[r]) & & X, Y \in \per \gL.
\end{align*}

\begin{cor} \label{c:RXrx}
Let $X \in \per \gL$. Then in $K_0(\per \gL)$ we have
\[
[R_i^{\gL}(X)] = [X] - \langle P_i, X \rangle_{\gL} [P_i] .
\]
\end{cor}
\begin{proof}
Since $R_i^{\gL}$ is triangulated, it is enough to verify this equality
on the basis elements $[P_j]$. This follows directly from Lemma~\ref{l:RiX}.
\end{proof}

The next lemma provides an explicit description of compositions of
functors $R^{\gL}_i$, which will be useful in the sequel.

\begin{lemma} \label{l:T}
Let $s \geq 1$ and let $\vphi : \{1, \dots, s\} \to \{1, \dots, n\}$ be
any function. Then
\[
R^{\gL}_{\vphi(s)} \cdot \ldots \cdot R^{\gL}_{\vphi(1)} = -
\tenl_{\gL} T^{\gL}_{\vphi}
\]
for the complex $T^{\gL}_{\vphi}$ of $\gL$-$\gL$-bimodules given by
\[
T^{\gL}_{\vphi} = \dots \to 0 \to T^{\gL,s}_{\vphi}
\xrightarrow{d^s_\vphi} \dots \to T^{\gL,r}_{\vphi}
\xrightarrow{d^r_\vphi} T^{\gL,r-1}_{\vphi} \to \dots
\xrightarrow{d^1_\vphi} T^{\gL, 0}_{\vphi} \to 0 \to \dots
\]
where
\begin{align} \label{e:Tr}
T^{\gL,0}_\vphi = \gL &,&  T^{\gL,r}_{\vphi} = \bigoplus_{1 \leq i_1 <
\dots < i_r \leq s} \gL e_{\vphi(i_1)} \otimes e_{\vphi(i_1)} \gL
e_{\vphi(i_2)} \otimes \dots \otimes e_{\vphi(i_r)} \gL
\end{align}
with the differentials $d^r_{\vphi}$ defined on each summand by
\begin{equation} \label{e:dTr}
d^r_{\vphi}(\lambda_0 \ten \lambda_1 \ten \dots \ten \lambda_r) =
\sum_{j=0}^{r-1} (-1)^j \lambda_0 \ten \dots \ten \lambda_j
\lambda_{j+1} \ten \dots \ten \lambda_r
\end{equation}
where $\lambda_0 \in \gL e_{\vphi(i_1)}$, $\lambda_r \in e_{\vphi(i_r)}
\gL$ and $\lambda_j \in e_{\vphi(i_j)} \gL e_{\vphi(i_{j+1})}$ for $0 <
j < r$.
\end{lemma}
\begin{proof}
By definition,
\begin{align*}
R^{\gL}_{\vphi(s)} \cdot \ldots \cdot R^{\gL}_{\vphi(2)} \cdot
R^{\gL}_{\vphi(1)} &= \Bigl( \dots \Bigl( \Bigl( - \tenl_{\gL}
C_{\vphi(1)} \Bigr) \tenl_{\gL} C_{\vphi(2)} \Bigr) \dots \tenl_{\gL}
C_{\vphi(s)} \Bigr) \\
&= - \tenl_{\gL} \Bigl( C_{\vphi(1)} \ten_{\gL} C_{\vphi(2)} \ten_{\gL}
\dots \ten_{\gL} C_{\vphi(s)} \Bigr)
\end{align*}
(where we replaced $\tenl$ by $\ten$ since the terms of $C_i$ are
projective as left (as well as right) modules), so it is enough to show
that
\[
T^{\gL}_{\vphi} = \Bigl( \dots \Bigl(C_{\vphi(1)} \ten_{\gL}
C_{\vphi(2)} \Bigr) \ten_{\gL} \dots \ten_{\gL} C_{\vphi(s)} \Bigr)
\]
where the right hand side is an iterated tensor product of complexes.

We prove this by induction on $s$, the case $s=1$ being merely the
definition of $R^{\gL}_{\vphi(1)}$. Now assume the claim for $s$,
consider a function $\vphi : \{1, \dots, s+1 \} \to \{1, \dots, n\}$
and denote by $\vphi'$ its restriction to $\{1, \dots, s\}$. By the
induction hypothesis, we need to show that $T^{\gL}_{\vphi} =
T^{\gL}_{\vphi'} \ten_{\gL} C_{\vphi(s+1)}$.

Recall that the tensor product of two complexes $X_{\gL}$ and $_{\gL}Y$
is defined by
\[
(X \ten_{\gL} Y)^m = \bigoplus_{p+q=m} X^p \ten_{\gL} Y^q
\]
with the differentials $d(x \ten y) = d(x) \ten y + (-1)^p x \ten d(y)$
for $x \in X^p$, $y \in Y^q$. It follows that for any $0 \leq r \leq
s+1$, the term at degree $-r$ of $T^{\gL}_{\vphi'} \ten_{\gL}
C_{\vphi(s+1)}$ equals
\[
T^{\gL,r}_{\vphi'} \oplus \left( T^{\gL,r-1}_{\vphi'} \ten_{\gL}
\left(\gL e_{\vphi(s+1)} \ten e_{\vphi(s+1)} \gL \right) \right)
\]
where the left summand vanishes for $r=s+1$ and the right vanishes for
$r=0$. Expanding these summands according to~\eqref{e:Tr}, we get a sum
over all the $r$-tuples $1 \leq i_1 < \dots < i_r \leq s+1$, where the
left summand corresponds to the tuples with $i_r \leq s$ while the
right to the tuples with $i_r = s+1$. Hence the term equals
$T^{\gL,r}_{\vphi}$.

Concerning the differentials, we have the following picture
\[
\xymatrix@C=0pt{
T^{\gL,r}_{\vphi} &=  &
T^{\gL,r}_{\vphi'} \ar[d]^{d^r_{\vphi'}} & \oplus &
T^{\gL,r-1}_{\vphi'} \ten_{\gL}
\left(\gL e_{\vphi(s+1)} \ten e_{\vphi(s+1)} \gL \right)
\ar[d]^{d^{r-1}_{\vphi'} \ten 1} \ar[dll]^(0.4){(-1)^{r-1} \ten m} \\
T^{\gL,r-1}_{\vphi} &= &
T^{\gL,r-1}_{\vphi'} & \oplus &
T^{\gL,r-2}_{\vphi'} \ten_{\gL}
\left(\gL e_{\vphi(s+1)}\ten e_{\vphi(s+1)} \gL \right)
}
\]
which shows that they coincide with the $d^r_{\vphi}$ as defined
in~\eqref{e:dTr}.
\end{proof}

As a side application, we show the following commutativity result which
is analogous to the fact that in a Weyl group corresponding to a
generalized Cartan matrix $B$, the two simple reflections $r^B_i$ and
$r^B_j$ commute when $B_{ij} = 0 = B_{ji}$, compare with
Proposition~2.12 of~\cite{SeidelThomas01}.

\begin{lemma}
If $\bform{P_i}{P_j}_{\gL} = 0 = \bform{P_j}{P_i}_{\gL}$ then
$R^{\gL}_i R^{\gL}_j \simeq R^{\gL}_j R^{\gL}_i$.
\end{lemma}
\begin{proof}
Indeed, $R^{\gL}_i R^{\gL}_j$ and $R^{\gL}_j R^{\gL}_i$ are given by
the complexes
\begin{align*}
& \gL e_j \ten e_j \gL e_i \ten e_i \gL \to
\left(\gL e_j \ten e_j \gL \right) \oplus
\left(\gL e_i \ten e_i \gL \right) \to \gL ,\\
& \gL e_i \ten e_i \gL e_j \ten e_j \gL \to
\left(\gL e_i \ten e_i \gL \right) \oplus
\left(\gL e_j \ten e_j \gL \right) \to \gL
\end{align*}
which are isomorphic since $e_j \gL e_i = 0 = e_i \gL e_j$.
\end{proof}

A special role is played by the composition $R_n^{\gL} \cdot \ldots
\cdot R_2^{\gL} \cdot R_1^{\gL}$ corresponding to the identity function
on $\{1,\dots,n\}$. We thus denote by $T^{\gL} = T^{\gL}_{id}$ the
corresponding complex of bimodules of Lemma~\ref{l:T}, so that
\begin{equation} \label{e:Tid}
R_n^{\gL} \cdot \ldots \cdot R_2^{\gL} \cdot R_1^{\gL} = - \tenl_{\gL}
T^{\gL} .
\end{equation}

\subsection{Triangular algebras}
%%%%%%%%%%%%%%%%%%%%%%%%%%%%%%%%
\label{ssec:triangular}

In this section we study the complexes $T^A$ for triangular algebras
$A$.

\begin{defn}
A finite dimensional algebra $A$ over a field $k$, with
primitive orthogonal idempotents $e_1, \dots, e_n$, is called
\emph{triangular} if $e_i A e_j = 0$ for all $j < i$ and $e_i A e_i
\simeq k$ for all $1 \leq i \leq n$.
\end{defn}

Triangular algebras have finite global dimension, hence the categories
$\per A$ and $\dA$ coincide.

\begin{lemma} \label{l:RiPj}
Let $A$ be triangular and let $1 \leq i \leq j \leq n$. Then
\[
R_i^A(P_j) \simeq \begin{cases} 0 & \text{if $j=i$,}
\\ P_j & \text{if $j > i$,}
\end{cases}
\]
in the homotopy category $\cK^b(A)$.
\end{lemma}
\begin{proof}
If $i < j$, then $\Hom_A(P_i, P_j) \simeq e_j A e_i = 0$,
hence by Lemma~\ref{l:RiX}, $R_i^A(P_j) = P_j$ (even in $\cC^b(A)$).

Similarly, $\Hom_A(P_i, P_i) \simeq k$, hence $R_i^A(P_i)$ equals the
null-homotopic complex $P_i \to P_i$, so it vanishes in $\cK^b(A)$.
\end{proof}

\begin{prop} \label{p:TA}
Let $A$ be triangular. Then:
\begin{enumerate}
\renewcommand{\theenumi}{\alph{enumi}}
\item
The functor $R_n^A \cdot \ldots \cdot R_2^A \cdot R_1^A$ on $\dA$ is
isomorphic to the zero functor.
\item
$T^A \simeq 0$ in $\cD^b(A^{op} \ten A)$.
\item
$T^A$ is contractible as a complex of right $A$-modules as well as a
complex of left $A$-modules.
\end{enumerate}
\end{prop}
\begin{proof}
A repeated application of Lemma~\ref{l:RiPj} shows that for $1 \leq j,
s \leq n$,
\[
(R^A_s \cdot \ldots \cdot R^A_1) (P_j) \simeq \begin{cases} 0 & \text{if $j
\leq s$,} \\
P_j & \text{if $j > s$,}
\end{cases}
\]
in $\cK^b(A)$, hence the complex $(R^A_n \cdot \ldots \cdot R^A_1)(A)$ is
homotopic to zero. Since $A$ generates $\dA$, the first assertion
follows. Now the second assertion follows from~\eqref{e:Tid}. For the
third, observe that all the terms of $T^A$ are projective both as right
and as left $A$-modules (in fact, the above argument shows directly the
contractibility of $T^A$ as a complex of right $A$-modules).
\end{proof}

\begin{rem}
Since all its terms at negative degrees are also projective as
$A$-$A$-bimodules, the complex $T^A$ yields a projective resolution of
$A$ as an $A$-$A$-bimodule, which can be useful when computing
Hochschild cohomology. Indeed, a similar resolution is given by
Cibils~\cite{Cibils89}, where an explicit contraction (of $k$-modules)
is also given.
\end{rem}

\begin{rem}
Since $T^A$ is contractible as a complex of left $A$-modules, the
tensor product $X \ten_A T^A$ yields a projective resolution of a right
module $X_A$. Similarly, $T^A \ten_A Y$ gives a projective resolution
of a left module $_AY$.
\end{rem}

The statement of Proposition~\ref{p:TA} is no longer true when the
assumption that $A$ is triangular is removed, even under the condition
that $A$ has finite global dimension. This is demonstrated by the
following example.

\begin{example}
Let $\gL$ be the path algebra of the quiver
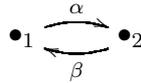
\[
\xymatrix{{\bullet_1} \ar@/^/[r]^{\alpha} & {\bullet_2}
\ar@/^/[l]^{\beta}}
\]
modulo the ideal generated by the path $\beta \alpha$. The algebra
$\gL$ is $5$-dimensional, its primitive orthogonal idempotents $e_1,
e_2$ correspond to the vertices of the quiver and its global dimension
is $2$. However, $\gL$ is not triangular as its Cartan matrix equals
\[
\begin{pmatrix}
2 & 1 \\ 1 & 1
\end{pmatrix} .
\]
Moreover, the complex
\[
T^{\gL} = \Bigl( \gL e_1 \ten e_1 \gL e_2 \ten e_2 \gL \to \left(\gL
e_1 \ten e_1 \gL \right) \oplus \left(\gL e_2 \ten e_2 \gL \right) \to
\gL \Bigr)
\]
is not acyclic since its Euler characteristic as a complex of vector
spaces (that is, the alternating sum of dimensions) is $3 \cdot 1 \cdot
2 - (3 \cdot 3 + 2 \cdot 2) + 5 \neq 0$.

Note that when $k=\bC$, the category $\modf \gL$ is equivalent to the
principal block of category $\cO$ of the complex Lie algebra
$\mathfrak{sl}_2$, see~\cite[\S5.1.1]{Stroppel03}.
\end{example}

For a triangular algebra $A$, the compositions of $R^A_i$ in the
\emph{reverse} order have a very simple form. This is recorded in the
next proposition, which will not be used in the sequel.

\begin{prop}
Let $A$ be triangular. Let $I \subseteq \{1,\dots,n\}$ and enumerate its
elements in decreasing order $I = \{i_1 > i_2 > \dots > i_s\}$. Then
\[
R^A_{i_s} \cdot \ldots \cdot R^A_{i_1} = - \tenl_A
\Bigl( \bigoplus_{i \in I} A e_i \ten e_i A \xrightarrow{m} A \Bigr)
\]
\end{prop}
\begin{proof}
Apply Lemma~\ref{l:T} for the function $\vphi$ defined by $\vphi(t) = i_t$
for $1 \leq t \leq s$ and observe that all the terms
$T^{A,r}_{\vphi}$ vanish when $r > 1$ as $e_{i_t} A e_{i_{t+1}} = 0$ for all
$1 \leq t < s$.
\end{proof}

\subsection{Triangular bimodules and their trivial extensions}
%%%%%%%%%%%%%%%%%%%%%%%%%%%%%%%%%%%%%%%%%%%%%%%%%%%%%%%%%%%%%%
Let $A$ be a basic finite-dimensional algebra with primitive orthogonal
idempotents $e_1, \dots, e_n$.

\begin{defn}
An $A$-$A$-bimodule $M$ is \emph{triangular} if $e_i M e_j = 0$ for all
$j < i$.
\end{defn}

Let $M$ be a triangular bimodule and let $DM = \Hom_k(M,k)$ be its
dual. Consider the \emph{trivial extension} $\gL = A \ltimes DM$, that is,
the $k$-algebra which has $A \oplus DM$ as an underlying vector
space, with the multiplication defined by $(a,\mu) (a',\mu') = (aa', a
\mu' + \mu a')$.

The ring homomorphisms $A \xrightarrow{\iota} \gL \xrightarrow{\pi} A$
given by
\begin{align*}
\iota(a) = (a,0) && \pi(a,\mu) = a
\end{align*}
give rise to the bimodules $_A{\gL}_A$ and $_{\gL}A_{\gL}$ (where $a
\in A$ acts via multiplication by $\iota(a)$ and $\lambda \in \gL$ acts
via multiplication by $\pi(\lambda)$). In particular we have the exact
functors
\begin{align*}
\iota^{*} &= - \ten_\gL \gL_A = \Hom_{\gL}({_A{\gL}_{\gL}}, -)
: \modf \gL \to \modf A \\
\pi^{*} &= - \ten_A A_{\gL} = \Hom_A ({_{\gL}A_A}, -)
: \modf A \to \modf \gL
\end{align*}
which induce functors
\begin{align*}
\dgL \xrightarrow{- \ten_{\gL} \gL_A} \dA
&, &
\dA \xrightarrow{- \ten_A A_{\gL}} \dgL .
\end{align*}

The left derived functors of their adjoints
\begin{align*}
- \ten_A \gL_{\gL} : \modf A \to \modf \gL
&, &
- \ten_{\gL} A_A : \modf \gL \to \modf A .
\end{align*}
give rise to
\begin{align*}
\dA = \per A \xrightarrow{- \tenl_A \gL_{\gL}} \per \gL
&, &
\per \gL \xrightarrow{- \tenl_{\gL} A_A} \per A = \dA .
\end{align*}

The elements $\iota(e_1),\dots,\iota(e_n)$ are primitive orthogonal
idempotents of $\gL$. We shall denote them by $e_1,\dots,e_n$ when
there is no risk of confusion.

\begin{prop} \label{p:SES}
Let $A$ be a finite-dimensional basic algebra and $M$ be a triangular
bimodule. Then there exist short exact sequences of complexes of
bi-modules
\begin{equation} \label{e:SES}
\begin{split}
0 \to {_{\gL}DM_A} \to \gL \ten_A T^A \to T^{\gL} \ten_{\gL} A \to 0 \\
0 \to {_ADM_{\gL}} \to T^A \ten_A \gL \to A \ten_{\gL} T^{\gL} \to 0
\end{split}
\end{equation}
\end{prop}
\begin{proof}
We prove only the exactness of the first sequence, as the proof for the other
is similar.

Let $1 \leq r \leq n$ and consider the terms at degree $-r$ of $\gL
\ten_A T^A$ and $T^{\gL} \ten_{\gL} A$ as direct sums
\begin{align*}
\bigl( \gL \ten_A T^A \bigr)^{-r} &= \bigoplus \gL e_{i_1} \ten e_{i_1}
A e_{i_2} \ten \dots \ten e_{i_{r-1}} A
e_{i_r} \ten e_{i_r} A \\
\bigl( T^{\gL} \ten_{\gL} A \bigr)^{-r} &= \bigoplus \gL e_{i_1} \ten
e_{i_1} \gL e_{i_2} \ten \dots \ten e_{i_{r-1}} \gL e_{i_r} \ten
e_{i_r} A
\end{align*}
running over the tuples $1 \leq i_1 < \dots < i_r \leq n$.

By our hypothesis that $M$ is a triangular bimodule, $e_j M e_i = 0$
hence $e_i DM e_j = 0$ for all $i < j$. Therefore we can identify
$e_i A e_j$ with $e_i \gL e_j$ (via either $\iota$ or $\pi$)
so that the terms $(\gL \ten_A T^A)^{-r}$
and $(T^{\gL} \ten_{\gL} A)^{-r}$ are isomorphic via the map
\[
\lambda_0 \ten a_1 \ten \dots \ten a_{r-1} \ten a_r \mapsto
\lambda_0 \ten \iota(a_1) \ten \dots \ten \iota(a_{r-1}) \ten a_r .
\]

Moreover, by considering the explicit forms of the right $A$-action on
$\gL$ and the left $\gL$-action on $A$,
\begin{align*}
\lambda_0 \cdot a_1 = \lambda_0 \iota(a_1) &,& a_{r-1}a_r =
\iota(a_{r-1}) \cdot a_r &,& \iota(a_j a_{j+1}) = \iota(a_j)
\iota(a_{j+1})
\end{align*}
for $1 \leq j < r-1$, we see that these isomorphisms commute with the
differentials as long as $r > 1$.

Finally, note that $(\gL \ten_A T^A)^0 = \gL$, $(T^{\gL} \ten_{\gL}
A)^0 = A$ and there is a commutative diagram
\[
\xymatrix{ {\bigoplus \gL e_i \ten e_i A} \ar[r]^(0.7){d^{A,1}}
\ar[d]^{\simeq} &
\gL \ar[d]^{\pi} \\
{\bigoplus \gL e_i \ten e_i A} \ar[r]^(0.7){d^{\gL,1}} & A }
\]
with the top and bottom differentials given by
\begin{align*}
d^{A,1} : \lambda_i \ten a_i \mapsto \lambda_i \iota(a_i) \in \gL &,&
d^{\gL,1} : \lambda_i \ten a_i \mapsto \pi(\lambda_i) a_i \in A
\end{align*}
respectively.

Summarizing, we get the following commutative diagram of complexes of
$A$-$\gL$-bimodules which shows the desired exact sequence.

{\small
\[
\xymatrix@=1pc{
& & & & 0 \ar[d] \\
& & & & DM \ar[d] \\
{\dots} \ar[r] &
\bigoplus \gL e_{i_1} \ten e_{i_1} A e_{i_2} \ten \dots \ten
e_{i_{r-1}} A e_{i_r} \ten e_{i_r} A \ar[r] \ar[d]^{\simeq} &
{\dots} \ar[r] &
\bigoplus \gL e_i \ten e_i A \ar[r] \ar[d]^{\simeq} &
\gL \ar[d]^{\pi} \\
{\dots} \ar[r] &
\bigoplus \gL e_{i_1} \ten e_{i_1} \gL e_{i_2} \ten \dots \ten
e_{i_{r-1}} \gL e_{i_r} \ten e_{i_r} A \ar[r] &
{\dots} \ar[r] &
\bigoplus \gL e_i \ten e_i A \ar[r] &
A \ar[d] \\
& & & & 0
}
\]
}
\end{proof}

\subsection{Proof of Theorem~\ref{t:AM}}
%%%%%%%%%%%%%%%%%%%%%%%%%%%%%%%%%%%%%%%%

Let $A$ be a triangular algebra with primitive orthogonal idempotents
$e_1, \dots, e_n$ and let $M$ be a triangular $A$-$A$-bimodule (with
respect to this ordering of the idempotents). In this case, we can
combine Propositions~\ref{p:TA} and~\ref{p:SES} and deduce the
following result.

\begin{cor} \label{c:QIS}
Let $\gL = A \ltimes DM$. We have
\begin{align*}
T^{\gL} \ten_{\gL} A \xrightarrow{\sim} DM[1] &&
A \ten_{\gL} T^{\gL} \xrightarrow{\sim} DM[1]
\end{align*}
in $\cD^b(\gL^{op} \otimes A)$ and $\cD^b(A^{op} \otimes \gL)$,
respectively.
\end{cor}
\begin{proof}
Since $T^A$ is contractible as a complex of left $A$-modules, the complex
$\gL \ten_A T^A$ is contractible as a complex of left $\gL$-modules, hence
it is isomorphic to zero in $\cD^b(\gL^{op} \ten A)$. Now the assertion
follows from the first short exact sequence in~\eqref{e:SES}.
The proof of the second assertion is similar.
\end{proof}

Part~(b) of Theorem~\ref{t:AM} now follows from Corollary~\ref{c:QIS}
by setting $R_i = R^{\gL}_i$ for $1 \leq i \leq n$ and
using~\eqref{e:Tid}.

\begin{rem}
When $M$ is zero, $\gL = A$ and we recover Proposition~\ref{p:TA}.
\end{rem}

\subsection{$K$-theoretic interpretation}
%%%%%%%%%%%%%%%%%%%%%%%%%%%%%%%%%%%%%%%%%

We now prove part~(a) of Theorem~\ref{t:AM} and explain how that
theorem can be regarded as a categorification of
equation~\eqref{e:prod}. In fact, we will recover this equation through
a process known as \emph{decategorification}, by looking at the effect
of the functors appearing in the theorem on the corresponding
Grothendieck groups.

Indeed, as the functors $R_i^{\gL}$, $- \tenl_A DM_A[1]$ and $-
\tenl_{\gL} A$ are triangulated, they induce linear maps on the
corresponding Grothendieck groups, which we describe explicitly in terms
of the so-called Cartan matrices of $A$ and $\gL$.

For an arbitrary (basic) finite dimensional algebra $\gL$ with
indecomposable projectives $P_1, \dots, P_n$, it will be convenient to
reorder them in reverse order and to consider the basis
\[
\eps_1 = [P_n], \eps_2 = [P_{n-1}], \dots, \eps_n = [P_1]
\]
of the Grothendieck group $K_0(\per \gL)$. We denote by $C_{\gL}$ the
matrix of the Euler form $\Bform_{\gL}$ with respect to that basis,
known as the \emph{Cartan matrix} of $\gL$. In explicit terms,
\begin{align*}
(C_{\gL})_{ij} &= \bform{P_{n+1-i}}{P_{n+1-j}}_{\gL} = \dim_k
\Hom_{\gL}(P_{n+1-i}, P_{n+1-j}) \\
&= \dim_k e_{n+1-j} \gL e_{n+1-i} .
\end{align*}

\begin{lemma}
Let $1 \leq i \leq n$. Then the matrix of the linear map on $K_0(\per
\gL)$ induced by $R_i^{\gL}$ is given by $r^{C_{\gL}}_{n+1-i}$.
\end{lemma}
\begin{proof}
The $j$-th column of that matrix is equal to the class of
$R^{\gL}_i(P_{n+1-j})$ in $K_0(\per \gL)$ which, according to
Corollary~\ref{c:RXrx}, equals
\[
[R^{\gL}_i(P_{n+1-j})] = [P_{n+1-j}] -
\bform{P_i}{P_{n+1-j}}_{\gL}[P_i] = \eps_j - (C_{\gL})_{n+1-i,j}
\eps_{n+1-i} .
\]
\end{proof}

For an algebra $A$ with primitive orthogonal idempotents $e_1, \dots,
e_n$, the condition that $A$ is triangular implies that the matrix
$C_A$ is upper triangular with ones on its main diagonal.
Similarly to the definition of $C_A$, one can define for any
$A$-$A$-bimodule $X$, a Cartan matrix $C_X$ by
\[
(C_X)_{ij} = \dim_k e_{n+1-j} X e_{n+1-i} ,
\]
so that $X$ is triangular if and only if $C_X$ is upper triangular.

\begin{lemma}
Let $A$ be a triangular algebra and $X$ an $A$-$A$-bimodule. Then the
matrix of the linear map on $K_0(\per A)$ induced by the functor $-
\tenl_A X$ is given by $C_A^{-1} C_X$.
\end{lemma}
\begin{proof}
Denote that matrix (with respect to the basis $\eps_1,\dots,\eps_n$) by
$x$. Since the functor $- \tenl_A X$ sends each $P_j$ to
$P_j \ten_A X \simeq e_j X$, we have
\[
[e_{n+1-j} X] = \sum_{i=1}^n x_{ij} [P_{n+1-i}]
\]
for all $1 \leq i \leq n$. Now, for any $1 \leq l \leq n$,
\begin{align*}
(C_X)_{lj} &= \dim_k e_{n+1-j} X e_{n+1-l} =
\bform{P_{n+1-l}}{e_{n+1-j}X}_A \\
&= \sum_{i=1}^n x_{ij} \bform{P_{n+1-l}}{P_{n+1-i}}_A = \sum_{i=1}^n
(C_A)_{li} x_{ij} ,
\end{align*}
hence $C_X = C_A x$.
\end{proof}

When $A$ is triangular and $M$ is a triangular bimodule, the Cartan
matrix of the trivial extension $\gL = A \ltimes DM$ equals $C_{\gL} =
C_A + C_{DM} = C_A + C_M^T$. Hence $(C_{\gL})_{+} = C_A$ is the upper
triangular part of $C_{\gL}$ and $(C_{\gL})_{-} = C_M$ is its lower
triangular part, as defined in Proposition~\ref{p:linalg}.

Combining everything together, observing that the functor $-
\tenl_{\gL} A$ sends the projective $\iota(e_i) \gL$ to $e_i A$ and
thus induces the identity matrix between the isomorphic groups $K_0(\per \gL)$
and $K_0(\per A)$, we conclude the following.

\begin{cor} \label{c:decat}
The left diagram of Theorem~\ref{t:AM} induces a commutative diagram on
the Grothendieck groups
\[
\xymatrix{
K_0(\per \gL) \ar[d]_{I_n} \ar[r]^{r^{C_{\gL}}_n} &
K_0(\per \gL) \ar[r]^(0.6){r^{C_{\gL}}_{n-1}} &
{\dots} \ar[r]^(0.4){r^{C_{\gL}}_1} &
K_0(\per \gL) \ar[d]^{I_n} \\
K_0(A) \ar[rrr]^{-(C_{\gL})_{+}^{-1} (C_{\gL})_{-}^T} & & & K_0(A)
}
\]
(where $I_n$ is the $n \times n$ identity matrix), thus
recovering equation~\eqref{e:prod}.
\end{cor}

\subsection{Proof of Corollary~\ref{c:Serre}}
%%%%%%%%%%%%%%%%%%%%%%%%%%%%%%%%%%%%%%%%%%%%%
Let $e_1,\dots,e_n$ be the primitive orthogonal idempotents of $A$ and
set $R_i = R^{T(A)}_i$ for $1 \leq i \leq n$.

The algebra $T(A)$ is symmetric and $\dim_k e_i T(A) e_i = 2$ for any
$1 \leq i \leq n$. Hence by~\cite[Remark~4.3]{HoshinoKato02}, the
functors $R^{T(A)}_i$ are autoequivalences, see
also~\cite[Theorem~4.1]{RouquierZimmermann03}.

Since $\nu_A = - \tenl_A DA$ and $A \ltimes DA = T(A)$,
Corollary~\ref{c:Serre} is just a special case of Theorem~\ref{t:AM},
where the triangular bimodule $M$ is taken to be $A$.

\begin{rem}
The Cartan matrix $B$ of $T(A)$ is symmetric with $2$ on its main
diagonal, hence the matrices $r^B_i$ are reflections. As the action of
each autoequivalence $R^{T(A)}_i$ on $K_0(\per T(A))$ is given by a
reflection, one may interpret this corollary as lifting of the Serre
functor (up to a shift by one) on $\dA$ to a product of ``reflection''
functors on $\per T(A)$.
\end{rem}

\subsection{Realization of matrices as Cartan matrices}
%%%%%%%%%%%%%%%%%%%%%%%%%%%%%%%%%%%%%%%%%%%%%%%%%%%%%%%
\label{ssec:realize}

We now show that Theorem~\ref{t:AM} categorifies~\eqref{e:prod} for any
integral square matrix $B$ with non-negative entries and positive
entries on its main diagonal. We start with an observation on the
realization of such matrices as Cartan matrices of finite dimensional
algebras.

\begin{lemma} \label{l:CasCA}
Let $C$ be an integral $n \times n$ matrix with $C_{ij} \geq 0$ and
$C_{ii} > 0$ for $1 \leq i, j \leq n$, and let $k$ be a field. Then
there exists a finite dimensional algebra $A$ over $k$ whose Cartan
matrix equals $C$.
\end{lemma}
\begin{proof}
We construct $A$ from a quiver with relations. Let $Q$ be the quiver
whose vertices are $\{1, 2, \dots, n\}$, with the number of arrows from
$i$ to $j$ set to
\[
\left|\{\text{arrows $i \to j$}\}\right| =
C_{n+1-j,n+1-i} - \delta_{ij} .
\]

Let $I \subseteq kQ$ be the ideal in the path algebra of $Q$ generated
by all the paths of length $2$. Then the Cartan matrix of $A = kQ/I$
equals $C$, since $\dim_k \Hom_A(P_{n+1-i}, P_{n+1-j})$ is the number
of paths in $Q$ from $n+1-j$ to $n+1-i$ of length at most one, which
equals $C_{ij}$ by construction.
\end{proof}

Observe that the algebra $A$ constructed in the lemma is triangular if
and only if $C$ is upper triangular with ones on its main diagonal. We
now consider the realization of bimodules with prescribed Cartan matrix
over such algebras.

\begin{lemma} \label{l:CasCM}
Let $C'$ be an integral $n \times n$ matrix with non-negative entries
and $A$ a finite dimensional algebra as constructed in the previous
lemma, which is furthermore triangular. Then there exists a bimodule
$M$ over $A$ such that $C_M = C'$.
\end{lemma}
\begin{proof}
Let $A = kQ/I$ be as in the previous lemma, and consider the quiver
$Q^{op} \times Q$ whose vertices are the pairs $(i,j)$ for $1 \leq i,j
\leq n$, with ``horizontal'' arrows $(i,j) \to (i,j')$ for every arrow
$j \to j'$ in $Q$ and ``vertical'' arrows $(i',j) \to (i,j)$ for every
arrow $i \to i'$ in $Q$. By our assumption on $A$, there are no loops
in this quiver.

We construct a representation $M$ of $Q^{op} \times Q$ as follows. For
any vertex $(i,j)$, let $M_{(i,j)}$ be a $k$-vector space of dimension
$C'_{n+1-j, n+1-i}$. For any arrow $\alpha : (i,j) \to (i',j')$, take
the linear map $M_{\alpha} : M_{(i,j)} \to M_{(i',j')}$ to be zero if
$i \neq j$ and arbitrary otherwise.

Observe that for any path in $Q^{op} \times Q$ of length $2$, the
corresponding map in $M$ vanishes, hence in particular $M$ can be
viewed as a bimodule over $A$. Moreover, $C_M = C'$ since
\[
(C_M)_{ij} = \dim_k e_{n+1-j} M e_{n+1-i} = \dim_k M_{(n+1-j,n+1-i)} =
C'_{ij}
\]
by construction.
\end{proof}

Note that in the proof we could have taken all the maps $M_{\alpha}$ to
be zero. However, the construction presented in the lemma is slightly
more general and in particular one can realize $A$ as a bimodule over
itself in this way.

Combining the above two observations, we deduce the following.

\begin{cor}
Let $B$ be an integral $n \times n$ matrix with $B_{ij} \geq 0$ and
$B_{ii} > 0$ for $1 \leq i, j \leq n$, and let $k$ be a field. Then
there exist a finite dimensional triangular algebra $A$ over $k$ and a
triangular bimodule $M$ over $A$ such that $C_A = B_+$ and $C_M = B_-$.
In particular, $B = C_{\gL}$ for $\gL = A \ltimes DM$.
\end{cor}
\begin{proof}
Use Lemma~\ref{l:CasCA} with $C=B_{+}$ to construct the algebra $A$, and
then Lemma~\ref{l:CasCM} with $C'=B_{-}$ to construct the bimodule $M$.
\end{proof}

In particular we see that Corollary~\ref{c:Serre}
categorifies~\eqref{e:coxprod} for any square upper triangular integer
matrix $C$ with non-negative entries and $1$ on its main diagonal (or
equivalently, for any square symmetric integer matrix $B$ with
non-negative entries and $2$ on its main diagonal).

\section{Discussion and Comparison}
%%%%%%%%%%%%%%%%%%%%%%%%%%%%%%%%%%%
\label{sec:quivers}

In this section we recall previous work on path algebras of quivers
that can be considered as a categorical interpretation of
equation~\eqref{e:coxprod}, and compare it with our approach.

\subsection{A result of Gabriel}
%%%%%%%%%%%%%%%%%%%%%%%%%%%%%%%%

Fix an algebraically closed field $k$. For a quiver $Q$ without
oriented cycles, denote by $kQ$ its path algebra and by $\dQ$ the
bounded derived category of finite-dimensional right\ $kQ$-modules.
Recall that a vertex $s \in Q$ is called a \emph{sink} if there are no
arrows in $Q$ starting at $s$. The \emph{reflection} of $Q$ at $s$,
denoted $\sigma_s Q$, is the quiver obtained from $Q$ by inverting all
the arrows ending at $s$ while leaving all the others intact, so that
$s$ becomes a \emph{source} in $\sigma_s Q$.

In~\cite{BGP73}, Bernstein, Gelfand and Ponomarev defined the
\emph{reflection functor} from the category of representations of $Q$
to those of $\sigma_s Q$ (where $s$ is a sink in $Q$). In the language
of derived categories (see for example~\cite[(IV.4,
Exercise~6)]{GelfandManin03}), this functor induces a derived
equivalence
\[
R_s : \dQ \xrightarrow{\sim} \cD^b(\sigma_s Q) .
\]

Order now the vertices of $Q$ in an \emph{admissible ordering}, that
is, enumerate them in a sequence $1,2,\dots, n$ such that there are no
arrows $j \to i$ in $Q$ for $i < j$. In this case, the vertex $1 \leq i
\leq n$ is a sink in the quiver $\sigma_{i+1} \sigma_{i+2} \dots
\sigma_n Q$. Moreover, the quiver $\sigma_1 \dots \sigma_n Q$ is equal
to $Q$. Thus, the composition of the (derived) reflection functors
\[
\dQ \xrightarrow{R_n} \cD^b(\sigma_n Q) \xrightarrow{R_{n-1}}
\cD^b(\sigma_{n-1} \sigma_n Q) \xrightarrow{R_{n-2}} \dots
\xrightarrow{R_1} \cD^b(\sigma_1 \dots \sigma_n Q)
\]
defines an autoequivalence $R_1 \cdot R_2 \cdot \ldots \cdot R_n$
of $\dQ$, known as the \emph{Coxeter functor}.

Another autoequivalence of $\dQ$ is given by the \emph{Auslander-Reiten
translation} $\tau$, which is related to the Serre functor $\nu = -
\tenl_{kQ} D(kQ)$ on $\dQ$ by $\tau = \nu[-1]$, see~\cite{Happel88}.
The following result of Gabriel~\cite{Gabriel80} relates it with the
Coxeter functor.

\begin{theorem}[\cite{Gabriel80}] \label{t:Gabriel}
If the underlying graph of $Q$ is a tree, or more generally, does not
contain a cycle of odd length, then
\begin{equation} \label{e:ARprod}
\tau \simeq R_1 \cdot R_2 \cdot \ldots \cdot R_n .
\end{equation}
\end{theorem}

Similarly to Corollary~\ref{c:decat}, the relation with
equation~\eqref{e:coxprod} is obtained through decategorification, by
considering the Grothendieck group $K_0(Q)$ of the triangulated
category $\dQ$ together with its Euler form $\Bform_{kQ}$, but this
time using bases of \emph{simple} modules rather than the
indecomposable projective ones.

Let $S_i$ be the simple module corresponding to the vertex $1 \leq i
\leq n$. The classes $[S_1], \dots, [S_n]$ form a basis of $K_0(Q)$,
and the matrix $C_Q$ of $\Bform_{kQ}$ with respect to that basis has an
explicit combinatorial description, namely
\[
(C_Q)_{ij} = \delta_{ij} -\bigl| \{\text{arrows } i \to j\} \bigr| .
\]
When the vertices are ordered in an admissible order, the matrix $C_Q$
is upper triangular with ones on its main diagonal.

It is well known that the matrix of the linear map on $K_0(Q)$ induced
by $\tau$ is given by $-C_Q^{-1} C_Q^T$. On the other hand, for a sink
$s$, the reflection functor $R_s$ induces a linear map $K_0(Q) \to
K_0(\sigma_s Q)$ whose matrix with respect to the bases of simples is
given by the reflection $r_s^{B_Q}$, where $B_Q = C_Q + C_Q^T$ is the
symmetrization of $C_Q$, see~\cite{BGP73}. Moreover, $B_{\sigma_s Q} =
B_Q$ since
\[
(B_Q)_{ij} = 2 \delta_{ij} - \bigl| \{\text{arrows } i \to j\} \bigr|
- \bigl|\{\text{arrows } j \to i\} \bigr|
\]
is independent on the orientation of the arrows.

Therefore, by looking at the Grothendieck groups,
Theorem~\ref{t:Gabriel} implies the following commutative diagram
\[
\xymatrix{
K_0(Q) \ar[r]^{r_n^{B_Q}} \ar[d]_{I_n} &
K_0(\sigma_n Q) \ar[r]^(0.6){r_{n-1}^{B_Q}} &
{\dots} \ar[r]^(0.35){r_1^{B_Q}} &
K_0(\sigma_1 \dots \sigma_n Q) \ar[d]^{I_n} \\
K_0(Q) \ar[rrr]^{-C_Q^{-1} C_Q^T} & & & {K_0(Q),} }
\]
recovering equation~\eqref{e:coxprod} for $C=C_Q$ as a $K$-theoretical
consequence of the isomorphism of the functors $\tau$ and
$R_1 \cdot R_2 \cdot \ldots \cdot R_n$ on $\dQ$.

\subsection{Comparison}
%%%%%%%%%%%%%%%%%%%%%%%
While Theorem~\ref{t:AM} and its Corollary~\ref{c:Serre}, on the one hand,
and Theorem~\ref{t:Gabriel}, on the other hand, can all be regarded as
categorical interpretations of equations~\eqref{e:prod}
and~\eqref{e:coxprod}, there are several differences,
which are outlined below.

\subsubsection{Scope}
Compared with Theorem~\ref{t:Gabriel}, Theorem~\ref{t:AM} has broader scope
in two aspects; first, it applies to any finite dimensional triangular
algebra $A$, and not only to hereditary ones. Second, it applies to any
triangular bimodule $M$, and not only to $M=A$, thus providing an
interpretation of equation~\eqref{e:prod} rather than the more
specific~\eqref{e:coxprod}.

\subsubsection{Lifting vs.\ factorization}
This broader scope carries some price to be paid, namely that while
Theorem~\ref{t:Gabriel} provides a factorization of the
Auslander-Reiten translation as a composition of the reflection
functors, Theorem~\ref{t:AM} does not factor $- \tenl_A DM[1]$, but
rather provides only a factorization of a lift to $\per \gL$ for $\gL =
A \ltimes DM$.

\subsubsection{The choice of matrix $C$}
Both Corollary~\ref{c:Serre} and Theorem~\ref{t:Gabriel} categorify
the same statement, namely equation~\eqref{e:coxprod}, and in both cases
the upper triangular matrix $C$ is the matrix of the Euler form with respect
to some basis. However, in Corollary~\ref{c:Serre} this is the basis of
indecomposable projectives, while in Theorem~\ref{t:Gabriel} it is the basis
of simples.

The use of the basis of simples is a rather special feature of
hereditary algebras. Indeed, recall that for a quiver $Q$ and a sink $s$, one
has $C_{\sigma_s Q} = r_s^T C_Q r_s$ where $r_s$ is the corresponding
reflection
built from the symmetrization of $C_Q$. However, for a general
triangular algebra $A$ whose Euler form is given by an upper triangular
matrix $C$ with respect to the basis of simples, there may be no
algebra whose Euler form is given by the matrix $r_s^T C r_s$, where $s$
is a sink or a source in the quiver of $A$ and $r_s$ is the corresponding
reflection built from the symmetrization of $C$.

\begin{example}
Let $A$ be the algebra given by the quiver
\[
\xymatrix{
{\bullet_1} \ar[r]^{\alpha} &
{\bullet_2} \ar[r]^{\beta} &
{\bullet_3}
}
\]
modulo the relation $\alpha \beta = 0$. The matrix of its
Euler form with respect to the basis of simples $\{S_1, S_2, S_3\}$ is
\[
C =
\begin{pmatrix}
1 & -1 & 1 \\
0 & 1 & -1 \\
0 & 0 & 1
\end{pmatrix} ,
\]
and
\begin{align*}
r_1 = \begin{pmatrix}
-1 & 1 & -1 \\ 0 & 1 & 0 \\ 0 & 0 & 1
\end{pmatrix}
&,&
r_2 = \begin{pmatrix}
1 & 0 & 0 \\ 1 & -1 & 1 \\ 0 & 0 & 1
\end{pmatrix}
&,&
r_3 = \begin{pmatrix}
1 & 0 & 0 \\ 0 & 1 & 0 \\ -1 & 1 & -1
\end{pmatrix}
\end{align*}
are the reflections built from the symmetrization $B = C + C^T$.

The matrices $r_1^T C r_1$ and $r_3^T C r_3$ cannot represent Euler
forms of algebras with respect to bases of simples. Indeed, if this
were the case then their inverses would be Cartan matrices of algebras,
which is impossible since
\begin{align*}
\left(r_1^T C r_1 \right)^{-1} = \begin{pmatrix}
1 & 0 & 0 \\ 0 & 1 & 1 \\ -1 & 0 & 1
\end{pmatrix}
&,&
\left(r_3^T C r_3 \right)^{-1} = \begin{pmatrix}
1 & 1 & 0 \\ 0 & 1 & 0 \\ -1 & 0 & 1
\end{pmatrix}
\end{align*}
contain negative entries.
\end{example}

\subsubsection{The shift}
Finally, observe that in both results, the minus sign in the
left hand side of~\eqref{e:prod} and~\eqref{e:coxprod} is interpreted as
a shift applied to the functor of tensoring with a bimodule.
However, in Theorem~\ref{t:AM} it is a positive
shift, while in Theorem~\ref{t:Gabriel} it is a negative one. Of course,
they are indistinguishable in the Grothendieck group.

%\bibliographystyle{amsplain}
%\bibliography{reflect}

\providecommand{\bysame}{\leavevmode\hbox to3em{\hrulefill}\thinspace}
\providecommand{\MR}{\relax\ifhmode\unskip\space\fi MR }
% \MRhref is called by the amsart/book/proc definition of \MR.
\providecommand{\MRhref}[2]{%
  \href{http://www.ams.org/mathscinet-getitem?mr=#1}{#2}
} \providecommand{\href}[2]{#2}

\end{document}